\documentclass[11pt]{amsart}

\usepackage{amsmath}
\usepackage{amssymb}
\usepackage{graphicx}
\allowdisplaybreaks

\theoremstyle{plain}
\newtheorem{thm}{Theorem}[section]
\newtheorem{lem}[thm]{Lemma}
\newtheorem{cor}[thm]{Corollary}

\theoremstyle{definition}
\newtheorem{defi}[thm]{Definition}

\numberwithin{equation}{section}

\numberwithin{equation}{section}

\def\Z{{\mathbb Z}}
\def\ls{\lesssim}
\def\gs{\gtrsim}
\def\X{\overline{X}}
\def\t{\widetilde{t}}
\def\l{\overline{l}}
\def\Corb{\operatorname{Corb}}
\def\Orb{\operatorname{Orb}}
\newtoks\by
\newtoks\paper
\newtoks\book
\newtoks\jour
\newtoks\yr
\newtoks\pages
\newtoks\vol
\newtoks\publ
\def\ota{{\hbox\vol{???}}}
\def\cLear{\by=\ota\paper=\ota\book=\ota\jour=\ota\yr=\ota
\pages=\ota\vol=\ota\publ=\ota}
\def\endpaper{\the\by, \textit{\the\paper},
{\the\jour} \textbf{\the\vol} (\the\yr), \the\pages.\cLear}
\def\endbook{\the\by, \textit{\the\book}, \the\publ.\cLear}
\def\endprep{\the\by, \textit{\the\paper}, \the\jour.\cLear}
\def\name#1#2{#1 #2}
\def\et{ and }

\begin{document}

\title[On the stability of the Lions-Peetre method of real interpolation] {On the stability of the Lions-Peetre method of  real interpolation with functional parameter}
\author{A. Gogatishvili}
\address{Amiran Gogatishvili, Institute of Mathematics of the Czech  Academy of Sciences, \u Zitn\'a 25, 11567 Prague 1, Czech Republic}
\email{gogatish@math.cas.cz}
\thanks{The research  was partly supported by the grant P201-18-0580S of the Grant Agency of the Czech Republic  by RVO: 67985840 and  Shota Rustaveli National Science Foundation (SRNSF), grant no: FR17-589 }

\begin{abstract}

Let $\vec{X}=(X_0, X_1)$ be a compatible couple of Banach spaces, $ 1\le p \le \infty$ and let $ \varphi$ be positive quasi-concave function. Denote by  $\overline{X}_{\varphi,p}=(X_0,X_1)_{\varphi,p}$ the real interpolation spaces defined by S. Janson (1981). We give necessary and sufficient conditions on $ \varphi_{0}$, $\varphi_{1}$ and $\varphi$ for the validity of
\begin{equation*}
	\left(\overline{X}_{\varphi_{0},1},\overline{X}_{\varphi_{1},1}
	\right) _{\varphi,p}= \left(\overline{X}_{\varphi_{0},\infty},\overline{X}_{\varphi_{1},\infty}\right)_{\varphi,p}
\end{equation*}
for all $ 1\le p\le \infty$, and all Banach couples $\overline{X}. $
\end{abstract}
\maketitle

\section{Introduction}
Let $\overline{X}=(X_0,X_1)$ be a compatible Banach couple. For $x\in \sum(  \overline{X}) =X_0+X_1$, 
 Peetre's  K-functional is defined by
$$ K(t,x, \overline{X})= \inf_ { x=x_0+x_1 , x_i\in X_i}
\left(\|x_0\|_{X_0}+t \|x_1\|_{X_1}\right) \quad t>0.
$$

Let $0<\theta<1$, $1\le p\le \infty$, then the Lions-Peetre spaces
$\overline{X}_{\theta,p}=(X_0,X_1)_{\theta,p}$   are defined using the norm
$$\|x\|_{\theta,p}=\left(\int^{\infty}_{0}\left(\frac{K(t,x, \overline{X})}{t^\theta}\right)^p\frac{dt}{t}\right)^{\frac{1}{p}}.
$$
 
One of the most important theoretical results for 
these spaces is the so-called  reiteration theorem, which claims that 
\begin{equation} \label{2} \left(\overline{X}_{\theta_0,p_0},\overline{X}_{\theta_1,p_1}\right)
_{\theta,p}=\overline{X}_{(1-\theta)\theta_0 +\theta \theta_1,p}, \quad \theta_0 \neq \theta_1.
\end{equation}

These definitions and properties can be found in any modern monograph on interpolation theory (e.g. \cite{BS}, \cite{BL} and \cite{T}).
The statement \eqref{2} is the  so called stability of the real method. The resulting space  on the  left-hand side  of \eqref{2} does not depend on $p_0$, $p_1$.

The definition of real interpolation method was extended in different directions by a number of authors. For example, one  can  replace the function $t^\theta$ in \eqref{2} by  a positive concave function $\varphi$, defined on $(0,\infty)$ (see T.F.Kalugina \cite{Kal}, J.Gustavsson \cite{Gu}, L.E.Persson \cite{P}). In \cite{J} S. Janson provided a different approach to these spaces using the discrete norm
$$\|x\|_{\varphi,p}=\left(\sum_{k\in \Z}\left(\frac{K(t_k,x, \overline{X})}{\varphi(t_k)}\right)^p\right)^{\frac{1}{p}},
$$
 where $\{t_k\}$ is a special discretizing sequence depending of $\varphi$ (see  Definition~\ref{ds}). In \cite[Theorem 19]{J} Janson  proved that, if $\varphi(\varphi_0, \varphi_1)$ and $\frac{\varphi_1}{\varphi_0}$ are quasi-power functions (see Definition~\ref{qpf}),  then the following  reiteration formula  holds for any $p_0, p_1, p\in [1,\infty]$, 
 \begin{align} \label{reit} \left(\overline{X}_{\varphi_0,p_0},\overline{X}_{\varphi_1,p_1}\right)
_{\varphi,p}=\overline{X}_{\varphi(\varphi_0,\varphi_1),p}\\
\left(\varphi(\varphi_0,\varphi_1)(t)=\varphi_0(t)\varphi\left(\frac{\varphi_1(t)}{\varphi_0(t)}\right)\right).\notag
\end{align}

 In \cite{Kr}, N.Krugljak gave  a necessary and sufficient condition
 on  $\varphi_0$, $\varphi_1$ and $\varphi$, so that \eqref{reit} is true for any choice of $p_0, p_1, p\in [1,\infty]$.     

A more general reiteration  theorem for a real interpolation method was obtained by S.Astashkin  \cite{As,As1},  Yu.A.Brudny\u\i and N.Ya.Krugljak \cite{BK}.

It is clear that, from \eqref{reit} we have
  \begin{equation} \label{stab} \left(\overline{X}_{\varphi_0,1},\overline{X}_{\varphi_1,1}\right)
_{\varphi,p}=\left(\overline{X}_{\varphi_0,\infty},\overline{X}_{\varphi_1,\infty}\right)_{\varphi,p}.
\end{equation}
 
 The inverse implication is not easy. The sufficient condition for \eqref{stab} was obtained by  E.Pustylnik \cite{Pu} and E.Semenov \cite{S} in case when $\varphi(t)=t^\theta$.  V.Ovchinnikov  \cite{Ov} offered  a new approach to study this problem.  Semenov and Ovchinnikov used Krugljak 's result \cite[Corollary 3]{AK}. Ovchinikov  only considers the case $\varphi(t)=t^\theta$.  In  this paper we are going extend  Ovchinnikov's theorem to te setting of non-degenerate quasi-concave function $\varphi$. In this context we will show (cf. Theorem~\ref{main})  that the reiteration theorem \eqref{reit}  follows  from the stability  theorem \eqref{stab}.  

In this paper we shall not consider the case of degenerate quasi-concave functions, which we leave as an open problem. We think that the study of the degenerate case  could be of interest to experts in Extrapolation Theory (see e.g. \cite{AsLyMi}. 
 
 
 We use the notation $A\ls B$ to indicate that  $A\le CB$ with some positive constant $C$ independent of appropriate quantities. If $A \ls B$  and $B\ls A$, we shall write $A\approx B$. 
 
 \section{Definitions and main result}
 We start with some basic definitions.
 
 \begin{defi} Let $\{a_k\}$ be a sequence of positive numbers. We shall say that $\{a_k\}$ is strongly  increasing (resp. strongly decreasing)  and write $a_{k}\uparrow \uparrow $ (resp.  $a_{k}\downarrow \downarrow $) if

\begin{equation*}
\underset{k\in \mathbb{Z}}{\inf }\frac{a_{k+1}}{a_{k}}\ge 2\quad \left(\text{ resp.  }\quad \underset{k\in \mathbb{Z}}{\sup }\frac{a_{k+1}}{a_{k}}\le\frac12,\right)
\end{equation*}

\end{defi} 

 \begin{defi} We shall say that  $\varphi$ is non-degenerate  quasi-concave function on $(0,\infty)$, if  $\varphi$ is non-decreasing and  $\frac{\varphi(t)}{t}$ is non-increasing on $(0,\infty)$ and, moreover,
\begin{equation}\lim_{t\to 0+}\varphi(t)=\lim_{t\to +\infty}\frac{\varphi(t)}{t}=\lim_{t\to 0+}\frac{t}{\varphi(t)}=\lim_{t\to +\infty}\frac{1}{\varphi(t)}=0.
\end{equation} 
\end{defi}
 
\begin{defi} \label{ds}
A strongly increasing sequence $\{t_k\}$ is a discretizing sequence for a  non-degenerate quasi-concave function $\varphi$ if 

i)  $\varphi(t_k)\uparrow\uparrow $ and $\frac{\varphi(t_{k+1})}{t_{k+1}}\downarrow\downarrow$ ;

ii) There exists a decomposition $\Z=\Z_1\cup\Z_2$, such that $\Z_1\cap\Z_2=\emptyset$, and 
\begin{align*}
\varphi(t_{k+1})& \le 2 \varphi(t_k) &\text{if}  \quad k\in \Z_1,\\ 
\frac{\varphi(t_k)}{t_k}& \le 2 \frac{\varphi(t_{k+1})}{t_{k+1}}
&\text{if}  \quad k\in \Z_2.
\end{align*} 
\end{defi}

Let us recall \cite[Lemma~2.7]{GP}, that if $\varphi$ is non-degenerate quasi-concave function  there always exists a discretizing sequence adapted to $\varphi$.  

 \begin{defi}\label{qpf} A quasi-concave function $\varphi$  is a quasi-power function ($\varphi\in P^{+-}$), whenever $s_{\varphi}(t)\rightarrow 0$ as $t\rightarrow 0$, $s_{\varphi}(t)\rightarrow \infty$ as $t\rightarrow \infty$, where $s_{\varphi}(t)=\sup_us_{\varphi}(ut)/s_{\varphi}(u) $  (cf. \cite{GuPe}). 
\end{defi}

It is known that (cf.  \cite{GuPe}), any quasi-power function $\varphi$ is equivalent to $t^{\theta_0}\psi(t^{\theta_1-\theta_0})$, for some quasi-concave function $\psi$ and $0<\theta_0,\theta_1<1$. If $\varphi$ is quasi-power function then  $\{2^k\}$ is a discretizing sequence for $\varphi$.

It is easy to see that if $\varphi_0$, $\varphi_1$ and $\varphi$ are non-degenerate positive quasi-concave functions on $(0,\infty)$, then the function $\varphi(\varphi_0,\varphi_1)(t)=\varphi_0(t)
\varphi(\frac{\varphi_1(t)}{\varphi_0(t)})$ is a non-degenerate quasi-concave function. Throughout  the paper the functions $\varphi_0$, $\varphi_1$ and $\varphi$ will be assumed to be non-degenerate positive quasi-concave functions on $(0,\infty)$ and  $\{ t_k\}$ (resp. $\{\widetilde t_k \}$) will denote the   discretizing sequence for $\varphi$ (resp. the discretizing sequence for  $\varphi(\varphi_0,\varphi_1)$).
 
\medskip
 Our main result now reads as  follows.

\begin{thm} \label{main} Let  $\varphi_0$, $\varphi_1$ and $\varphi$ be positive non-degenerate quasi-concave functions on $(0,\infty)$.  Let $\{ t_k\}$ be  discretizing sequence for  $\varphi$ and let  $\{\widetilde t_k \}$ be  discretizing sequence for  $\varphi(\varphi_0,\varphi_1)$. 
The following assertions are equivalent:

 {\rm(i)}  \ \eqref{stab} holds  for some $p\in [1,\infty]$;
 
 {\rm(ii)} \ \eqref{stab} holds for every $p\in [1,\infty]$;
 
 {\rm(iii)} \eqref{stab} holds for $\X=(L_1(0,\infty), L_\infty(0,\infty))$ and some $p\in [1,\infty]$;
 
	{\rm(iv)} \eqref{reit} holds for any  $p_0, p_1, p \in [1,\infty]$;
	
	{\rm(v)} $ \sup_{n\in \Z} \text{Card}\{k\in \Z :\quad t_n \le \frac{\varphi_0(\widetilde t_k)}{\varphi_1(\widetilde t_k)}\le t_{n+1}\}<\infty.$
\end{thm}

\section{Descriptions of some  special interpolation spaces} 

\begin{defi}Suppose that  $X$ is an intermediate space for a compatible couple $(X_0, X_1)$.
The orbit of the space $X$ relative to linear bounded operators mapping  the  couple  $\{ X_0,X_1\}$ to the couple $\{ Y_0,Y_1\}$, which will be denoted by 
$$\Orb\left(X, \{ X_0,X_1\}\to \{ Y_0,Y_1\}\right),$$
is the  linear space of all  $y\in Y_0+Y_1$  that can be represented  by  
$$ y=\sum_{j=1}^{\infty} T_jx_j, \quad \text{convergence in } \quad Y_0+Y_1,  $$ 
where $$\sum_{j=1}^{\infty} \max(\|T_j\|_{X_0\to Y_0},\|T_j\|_{X_1\to Y_1})\|x_j\|_{X}  < \infty .$$
\end{defi}
\begin{defi}Suppose $X$ is an intermediate space for a compatible couple $(X_0, X_1)$.
The coorbit of the space $Y\in Y_0+Y_1$ relative to linear bounded operators mapping  the  couple  $\{ X_0,X_1\}$ to the couple $\{ Y_0,Y_1\}$,  which will be denoted by  
$$\Corb\left(Y, \{ X_0,X_1\}\to \{ Y_0,Y_1\}\right),$$
is the linear space of all $x\in X_0+X_1$ such that 
$$ \sup  \{ \|T(x)\|_{Y}:\quad \|T\|_{\{ X_0,X_1\}\to \{ Y_0,Y_1 \}}\le 1\}<\infty.$$ 
\end{defi}

If $E$ is a sequence space and $\{w_k\}$ is a positive sequence (a weight), then $E(w)$ denotes the space of sequences $\{a_k\}$
such that $ a_kw_k \in E$, provided with its natural norm $\|a\|_{E(w)}=\|a_kw_k\|_{E}$. 

Let $\overline{l}_q=(l_q, l_q(\frac{1}{\t_k}))$, $ q\in [1, \infty]$   be a Banach couple consisting  of two sided infinite sequences, where $\{\t_k \}$ is  discretizing sequence for $\varphi(\varphi_0,\varphi_1)$.  
The main property of real method in terms of Orbits can be formulated as follows (see \cite{J})
\begin{align} \label{orb}
\X_{\varphi,p} &=
\Orb\left(l_p\left(\frac{1}{\varphi(\varphi_0,\varphi_1)(\t_k)}\right), \overline{l}_1 \to \X\right) \\
& = \Corb\left(l_p\left(\frac{1}{\varphi(\varphi_0,\varphi_1)(\t_k)}\right), \X \to \overline{l}_\infty\right).\notag
\end{align}
 As it is known (see \cite{J}) 
 \begin{equation}\label{l1p}
(\l_1)_{\varphi(\varphi_0,\varphi_1),p}=(\l_\infty)_{\varphi(\varphi_0,\varphi_1),p}=l_p\left(\frac{1}{\varphi(\varphi_0,\varphi_1)(\t_k)} \right).
 \end{equation} 
 by embeddings $ l_1\subset  l_q \subset l_\infty $ and 
 $  l_1\left(\frac 1{\t_k}\right)\subset  l_q \left(\frac 1{\t_k}\right)\subset l_\infty\left(\frac 1{\t_k}\right) $ we get 
 \begin{equation}\label{lqp}
(\l_q)_{\varphi(\varphi_0,\varphi_1),p}=l_p\left(\frac{1}{\varphi(\varphi_0,\varphi_1)(\t_k)} \right).
 \end{equation}  
 
Let   $l_q^{M}$  be  the space of sequences with indices in the set $M$. So if $E=l_p\left(l_q^{M_k}\right)$ then 
$$ \|a\|_{E}=\left(\sum_{k\in \Z}\left(\sum_{i\in M_k}|a_i|^q\right)^{p/q}\right)^{1/p}.
$$  

The next lemma is a  functional parameter version of Gilbert's interpolation theorem 
\cite{Gi} 
 \begin{lem}\label{gg}
 Let $p,q\in [1, \infty]$. Let  $\{v_k\}$, $\{w_k\}$ be positive sequences and $\varphi$ be non-degenerate quasi-concave function on $(0,\infty)$ then
 \begin{equation} \label{lvlw}
 \left(l_q(v),l_q(w)\right)_{\varphi,p}=
l_p\left(l_q^{M_k}\left(\frac{v_i}{\varphi\left(\frac{v_i}{w_i}\right)}\right)\right),
\end{equation}
where $M_k= \{ i: t_{k-1}< \frac{v_i}{w_i} \le t_k\}$ and 
 $\{ t_k\}$ is  a  discretizing sequence for $\varphi$.
\end{lem}
\begin{proof} Let $ E=\left(l_q(v),l_q(w)\right)_{\varphi,p}$. Assume that $p,q<\infty$.
Since 
$$ K(t,\{a_k\};\{ l_q(v),l_q(w)\})\approx 
\left( \sum_{k\in\Z} |a_k|^q\min{(v_k,tw_k)}^q\right)^{1/q}.
$$
it follows that 
$$ \|\{a_k\}\|_E\approx \left(\sum _{i\in \Z}\frac{\left( \sum_{k\in\Z} |a_k|^q\min{(v_k,t_iw_k)}^q\right)^{p/q}}{\varphi(t_i)^p}\right)^{1/p}.
$$
By the definition of of discretizing sequence (cf. Definition~\ref{ds}), we obtain
\begin{align}\label{3}
 \|\{a_k\}\|_E &\gs \left(\sum _{i\in \Z_1}\frac{\left( \sum_{t_{i-1}<\frac{v_k}{w_k}\le t_i} |a_k|^q v_k^q\right)^{p/q}}{\varphi(t_i)^p}\right)^{1/p} \\
 &\hskip+0.5cm + \left(\sum _{i\in \Z_2}\frac{t_{i-1}^p\left( \sum_{t_{i-1}<\frac{v_k}{w_k}\le t_i} |a_k|^q w_k^q\right)^{p/q}}{\varphi(t_{i-1})^p}\right)^{1/p}\notag \\
 &\approx \left(\sum _{i\in \Z}\left( \sum_{t_{i-1}<\frac{v_k}{w_k}\le t_i} |a_k|^q \left(\frac{v_k}{\varphi\left(\frac{v_k}{w_k}\right)}\right)^q\right)^{p/q}\right)^{1/p}.  \notag
 \end{align} 
  To prove the reverse estimate we use Lemma~3.1 from \cite{GP}.
  \begin{align}\label{4}
 \|\{a_k \}\|_{E} &\approx \left(\sum _{i\in \Z}\frac{\left( \sum_{j\in\Z} \sum_{t_{j-1}<\frac{v_k}{w_k}\le t_j} |a_k|^q \min{(v_k,t_iw_k)}^q\right)^{p/q}}{\varphi(t_i)^p}\right)^{1/p} \\
 &\ls \left(\sum _{i\in \Z}\frac{\left(\sum_{j\le i} \sum_{t_{j-1}<\frac{v_k}{w_k}\le t_j} |a_k|^q v_k^q\right)^{p/q}}{\varphi(t_{i})^p}\right)^{1/p}\notag \\
 &\hskip+0.5cm + \left(\sum _{i\in \Z}\frac{t_i^p\left(\sum_{j> i} \sum_{t_{j-1}<\frac{v_k}{w_k}\le t_j} |a_k|^q w_k^q\right)^{p/q}}{\varphi(t_{i})^p}\right)^{1/p}\notag \\ 
 &\approx \left(\sum _{i\in \Z}\frac{\left( \sum_{t_{i-1}<\frac{v_k}{w_k}\le t_i} |a_k|^q v_k^q\right)^{p/q}}{\varphi(t_{i})^p}\right)^{1/p}\notag \\
 &\hskip+0.5cm + \left(\sum _{i\in \Z}\frac{t_i^p\left( \sum_{t_{i}<\frac{v_k}{w_k}\le t_{i+1}} |a_k|^q w_k^q\right)^{p/q}}{\varphi(t_{i})^p}\right)^{1/p}\notag \\ 
 & \ls \left(\sum _{i\in \Z}\left( \sum_{t_{i-1}<\frac{v_k}{w_k}\le t_{i}} |a_k|^q \left(\frac{v_k}{\varphi\left(\frac{v_k}{w_k}\right)}\right)^q\right)^{p/q}\right)^{1/p}. \notag
 \end{align}   
 Therefore from \eqref{3} and \eqref{4} we obtain \eqref{lvlw}. The case  $p=\infty$ or $q=\infty$  can be obtained using  a similar argument and we shall omit the details. 

The proof is complete.
\end{proof}
\begin{cor} \label{cor1} Let $p,q\in [1,\infty]$.
Let $\{\t_k\}$ be  discretizing sequence for $\varphi(\varphi_0,\varphi_1)$ and let $\{t_k\}$ be  discretizing sequence for $\varphi$. Then  
$$\left(l_q\left(\frac1{\varphi_0(\t_k)}\right), l_q\left(\frac1{\varphi_1(\t_k)}\right)\right)_{\varphi,p}=l_p\left(l_q^{M_k}\left(\frac{1}{\varphi(\varphi_0,\varphi_1)(\t_i)}\right)\right),
$$
where  $M_k =\{ i: t_k \le \frac{\varphi_1(\t_i)}{\varphi_0(\t_i)}\le t_{k+1}\}$. 
\end{cor}
\begin{cor} \label{cor2} Let $p,q\in [1,\infty]$.
Let $\{ \tau_k\}$  be  discretizing sequence for $\varphi_0$, 
let  $\{ z_k\}$ be  discretizing sequence for $\varphi_1$ and  let  $\{\t_k\}$ be  discretizing sequence for $\varphi(\varphi_0,\varphi_1)$ . Then,
\begin{align}(\l_q)_{\varphi_0,p}&=l_p\left(l_q^{M_k^0}\left(\frac{1}{\varphi_0(\t_i)}\right)\right),\\
(\l_q)_{\varphi_1,p}&=l_p\left(l_q^{M_k^1}\left(\frac{1}{\varphi_1(\t_i)}\right)\right),
\end{align}
where  $M_k^0 =\{ i: \tau_k \le \t_i\le \tau_{k+1}\}$ and  
$M_k^1 =\{ i: z_k \le \t_i\le z_{k+1}\}$. 
\end{cor}
\begin{cor} \label{cor3}Let $p\in [1,\infty]$. Then
$$
\left((\l_p)_{\varphi_0,p},(\l_p)_{\varphi_1,p}\right)_{\varphi,p}   =l_p\left(\frac{1}{\varphi(\varphi_0,\varphi_1)(\t_i)}\right).
$$ 
\end{cor}
\begin{proof} Using Lemma~\ref{gg} we get
\begin{align*}
\left((\l_p)_{\varphi_0,p},(\l_p)_{\varphi_1,p}\right)_{\varphi,p}
&=\left(l_p\left(\frac1{\varphi_0(\t_k)}\right), l_p\left(\frac1{\varphi_1(\t_k)}\right)\right)_{\varphi,p}\\
&=l_p\left(\frac{1}{\varphi(\varphi_0,\varphi_1)(\t_i)}\right).
\end{align*}
\end{proof}
\begin{lem} \label{summax1}
Let $r\in {\mathbb R}_1 \setminus \{0\}$.
Let $\{ \tau_k\}$ be  discretizing sequence for $\varphi_0$
and let $\{ z_k\}$ be discretizing sequence for $\varphi_1$. Then we have the following estimates:
\begin{align} \label{sum<max1}	
	\sum_{\tau_k \le \t_i \le \tau_{k+1}} \left(\frac{\varphi_1(\t_i)}{\varphi_0(\t_i)}\right)^r 
	&\ls \sup_{\tau_k \le \t_i \le \tau_{k+1}} \left(\frac{\varphi_1(\t_i)}{\varphi_0(\t_i)}\right)^r, \\
	 \label{sum<max2}	\sum_{z_k \le \t_i \le z_{k+1}} \left(\frac{\varphi_1(\t_i)}{\varphi_0(\t_i)} \right)^r 
		&\ls \sup_{z_k \le \t_i \le z_{k+1}} \left(\frac{\varphi_1(\t_i)}{\varphi_0(\t_i)}\right)^r
	\end{align}
with constants independent of $\tau_k$ and  $z_k$.
\end{lem}
\begin{proof}
Let us fix $k$. We consider two possibilities. Either  
\begin{align}\label{Z_1}\varphi_0(\tau_{k+1})&\le 2 \varphi_0(\tau_{k}) 
\intertext{or} 
\frac{\varphi_0(\tau_{k})}{\tau_k}&\le 2 \frac{\varphi_0(\tau_{k+1})}{\tau_{k+1}}.
 \label{Z_2}
 \end{align}
Suppose that \eqref{Z_1} holds. In this case we will show that
  $$\frac{\varphi_1(\t_{i+2})}{\varphi_0(\t_{i+2})}\ge \frac{3}{2}
\frac{\varphi_1(\t_{i})}{\varphi_0(\t_{i})}$$
for every $\t_i, \t_{i+2} \in [\tau_k, \tau_{k+1}]$.
Indeed, if we assume that $$\frac{\varphi_1(\t_{i+2})}{\varphi_0(\t_{i+2})}\le \frac{3}{2}
\frac{\varphi_1(\t_{i})}{\varphi_0(\t_{i})},$$
 by using \eqref{Z_1} we get that 
 $$\varphi_0(\t_{i+2})\varphi\left(
 \frac{\varphi_1(\t_{i+2})}{\varphi_0(\t_{i+2})}\right)\le 2\varphi_0(\t_{i})\varphi\left(\frac{3}{2}
 \frac{\varphi_1(\t_{i})}{\varphi_0(\t_{i})}\right)\le 3\varphi_0(\t_{i})\varphi\left(
 \frac{\varphi_1(\t_{i})}{\varphi_0(\t_{i})}\right),$$
as $\{\t_k\}$ is discretizing sequence for $\varphi(\varphi_0,\varphi_1)$  we obtain the  contradiction.  
Suppose now that  \eqref{Z_2} holds. In this case we will show that
  $$\frac{\varphi_1(\t_{i+2})}{\varphi_0(\t_{i+2})}\le \frac{2}{3}
\frac{\varphi_1(\t_{i})}{\varphi_0(\t_{i})}$$
for every $\t_i, \t_{i+2} \in [\tau_k, \tau_{k+1}]$.
Indeed, if we assume that $$\frac{\varphi_1(\t_{i+2})}{\varphi_0(\t_{i+2})}\ge \frac{2}{3}
\frac{\varphi_1(\t_{i})}{\varphi_0(\t_{i})},$$
 then using \eqref{Z_2} we get that 
\begin{align*}\frac{\varphi_0(\t_{i})}{\t_i}\varphi\left(
 \frac{\varphi_1(\t_{i})}{\varphi_0(\t_{i})}\right)&\le 2\frac{\varphi_0(\t_{i+2})}{\t_{i+2}}\varphi\left(\frac{3}{2}
 \frac{\varphi_1(\t_{i+2})}{\varphi_0(\t_{i+2})}\right)\\
 &\le 3\frac{\varphi_0(\t_{i+2})}{\t_{i+2}}\varphi\left(
 \frac{\varphi_1(\t_{i+2})}{\varphi_0(\t_{i+2})}\right),
 \end{align*}
which, once again, is  a contradiction.
  Therefore, we obtain
  \begin{align*}
	 \sum_{\tau_k \le \t_i \le \tau_{k+1}} \left(\frac{\varphi_1(\t_i)}{\varphi_0(\t_i)}\right)^{r}
	&\ls \sup_{\tau_k \le \t_i \le \tau_{k+1}} \left(\frac{\varphi_1(\t_i)}{\varphi_0(\t_i)}\right)^{r}	\sum_{i=1}^{\infty} \left(\frac 23\right)^{|r|i}\\
  	&\ls \sup_{\tau_k \le \t_i \le \tau_{k+1}} \left(\frac{\varphi_1(\t_i)}{\varphi_0(\t_i)}\right)^{r}.
  	\end{align*}
In a similar fashion we can show the estimate \eqref{sum<max2}.
\end{proof}

\begin{lem} Let $r\in {\mathbb R}_1 \setminus\{ 0\}$.
Let $\{ \tau_k\}$ be  discretizing sequence for $\varphi_0$
and $\{z_k\}$ be discretizing sequence for $\varphi_1$. Then 
\begin{align} \label{sum<max3}
	\sum_{\tau_k \le \t_i \le \tau_{k+1}} \varphi\left(\frac{\varphi_1(\t_i)}{\varphi_0(\t_i)}\right)^r &\ls \sup_{\tau_k \le \t_i \le \tau_{k+1}} \varphi\left(\frac{\varphi_1(\t_i)}{\varphi_0(\t_i)}\right)^r, \\
		\sum_{z_k \le \t_i \le z_{k+1}}     \label{sum<max4} \left(\frac{\varphi_0(\t_i)}{\varphi_1(\t_i)} \varphi\left(\frac{\varphi_1(\t_i)}{\varphi_0(\t_i)}\right)\right)^r &\ls \sup_{z_k \le \t_i \le z_{k+1}} \left(\frac{\varphi_0(\t_i)}{\varphi_1(\t_i)}
\varphi\left( \frac{\varphi_1(\t_i)}{\varphi_0(\t_i)}\right)\right)^r
	\end{align}
	with constants independent of $\tau_k$ and $z_k$.
	\end{lem}
\begin{proof}
Let us fix $k$. As it was mentioned during the course of the proof of Lemma~\ref{summax1}.
we have two cases either  \eqref{Z_1} or   \eqref{Z_2}.
If \eqref{Z_1} holds we get 
  \begin{align*}
	\sum_{\tau_k \le \t_i \le \tau_{k+1}} \varphi\left(\frac{\varphi_1(\t_i)}{\varphi_0(\t_i)}\right)^r 
	&\ls \frac{1}{\varphi_0(\tau_k)^r} 
		\sum_{\tau_k \le \t_i \le \tau_{k+1}} \varphi_0(\t_i)^r\varphi\left(\frac{\varphi_1(\t_i)}{\varphi_0(\t_i)}\right)^r\\
	&\ls \frac{1}{\varphi_0(\tau_k)^r} 
		\sup_{\tau_k \le \t_i \le \tau_{k+1}} \varphi_0(\t_i)^r\varphi\left(\frac{\varphi_1(\t_i)}{\varphi_0(\t_i)}\right)^r \\
	&\ls \sup_{\tau_k \le \t_i \le \tau_{k+1}} \varphi\left(\frac{\varphi_1(\t_i)}{\varphi_0(\t_i)}\right)^r.	
\end{align*}
If \eqref{Z_2} holds we get
  \begin{align*}
	\sum_{\tau_k \le \t_i \le \tau_{k+1}} \varphi\left(\frac{\varphi_1(\t_i)}{\varphi_0(\t_i)}\right)^r 
	&\ls \left(\frac{\tau_k}{\varphi_0(\tau_k)}\right)^r 
		\sum_{\tau_k \le \t_i \le \tau_{k+1}} \left(\frac{\varphi_0(\t_i)}{\t_i}\right)^r
		\varphi\left(\frac{\varphi_1(\t_i)}{\varphi_0(\t_i)}\right)^r\\
	&\ls \left(\frac{\tau_k}{\varphi_0(\tau_k)}\right)^r 
		\sup_{\tau_k \le \t_i \le \tau_{k+1}} \left(\frac{\varphi_0(\t_i)}{\t_i}\right)^r
		\varphi\left(\frac{\varphi_1(\t_i)}{\varphi_0(\t_i)}\right)^r \\
	&\ls \sup_{\tau_k \le \t_i \le \tau_{k+1}} \varphi\left(\frac{\varphi_1(\t_i)}{\varphi_0(\t_i)}\right)^r.	
\end{align*}
The proof of  the estimate \eqref{sum<max3} is complete.
The proof of \eqref{sum<max4} is similar and we omit the details.
\end{proof}

\begin{lem} \label{kfunc}Let $q,p \in[1,\infty]$. Let $\{ \tau_k\}$  be  discretizing sequence for $\varphi_0$ and Let
 $\{z_k\}$ be  discretizing sequence for $\varphi_1$. 
 Let $M_k^0 =\{ i: \tau_k \le \t_i\le \tau_{k+1}\}$ and  
$M_k^1 =\{ i: z_k \le \t_i\le z_{k+1}\}$. 
Denote  sets 
$ \Omega_t=\{ y:  \frac{\varphi_1(y)}{\varphi_0(y)}\le t\}$ and 
$ \Omega_t^c=\{ y:  \frac{\varphi_1(y)}{\varphi_0(y)}> t\}$. Then
\begin{align*}
 &K\left(\{a_k\},t; l_p\left(l_q^{M_k^0}\left(\frac{1}{\varphi_0(\t_i)}\right)\right), l_p\left(l_q^{M_k^1}\left(\frac{1}{\varphi_1(\t_i)}\right)\right)
    \right)\\
&\hskip+1cm\approx
\|a_i\chi_{\Omega_t}(\t_i)\|_{l_p\left(l_q^{M_k^0}\left(\frac{1}{\varphi_0(\t_i)}\right)\right)}+t\|a_i\chi_{\Omega_t^c}(\t_i)\|_{l_p\left(l_q^{M_k^1}\left(\frac{1}{\varphi_1(\t_i)}\right)\right)}.
\end{align*}
\end{lem}
\begin{proof} Let us  consider sequences  $b_i=a_i\chi_{\Omega_t}(\t_i)$
and $c_i=a_i\chi_{\Omega_t^c}(\t_i)$. Hence  $a_i=b_i+c_i$.
Using the definition of $K$-functional it easy to see that 
 \begin{align*}
     & K\left(\{a_k\},t; l_p\left(l_q^{M_k^0}\left(\frac{1}{\varphi_0(\t_i)}\right)\right), l_p\left(l_q^{M_k^1}\left(\frac{1}{\varphi_1(\t_i)}\right)\right)
    \right)\\
    &\hskip+1cm\ls
 \left(\sum_{k\in\Z}\left(
\sum_{\t_i\in M_k^0, \frac{\varphi_1(\t_i)}{\varphi_0(\t_i)}\le t}
\left(\frac{|a_i|}{\varphi_0(\t_i)}\right)^q\right)^{\frac{p}{q}}\right)^{1/p}  \\
  &\hskip+1.5cm +
 t \left(\sum_{k\in\Z}\left(
\sum_{\t_i\in M_k^1, \frac{\varphi_1(\t_i)}{\varphi_1(\t_i)}> t}
\left(\frac{|a_i|}{\varphi_1(\t_i)}\right)^q\right)^{\frac{p}{q}}\right)^{1/p}\\
&\hskip+1cm =\|a_i\chi_{\Omega_t}(\t_i)\|_{l_p\left(l_q^{M_k^0}\left(\frac{1}{\varphi_0(\t_i)}\right)\right)}+t\|a_i\chi_{\Omega_t^c}(\t_i)\|_{l_p\left(l_q^{M_k^1}\left(\frac{1}{\varphi_1(\t_i)}\right)\right)}.
\end{align*}
So we have obtained an upper bound for the $K$-functional.

Let
$\{a_k\}\in l_p  \left(l_q^{M_k^0}\left(\frac{1}{\varphi_0(\t_i)}\right)\right)+ l_p\left(l_q^{M_k^1}\left(\frac{1}{\varphi_1(\t_i)}\right)\right)$ and let us  consider  any representation 
$a_i=b_i+c_i$, $i\in Z$ with $\{b_i\}\in l_p\left(l_q^{M_k^0}\left(\frac{1}{\varphi_0(\t_i)}\right)\right)$ and $ \{c_i \}\in l_p\left(l_q^{M_k^1}\left(\frac{1}{\varphi_1(\t_i)}\right)\right)$.
By using estimates \eqref{sum<max1} and \eqref{sum<max2} we obtain 
\begin{align*}
&\left(\sum_{k\in\Z}\left(
\sum_{\t_i\in M_k^0, \frac{\varphi_1(\t_i)}{\varphi_0(\t_i)}\le t}
\left(\frac{|c_i|}{\varphi_0(\t_i)}\right)^q\right)^{\frac{p}{q}}\right)^{1/p} \\
 &\hskip+1cm\ls  \left(\sum_{k\in\Z}
\left(
\sum_{\t_i\in M_k^0, \frac{\varphi_1(\t_i)}{\varphi_0(\t_i)}\le t}
\left(\frac{\varphi_1(\t_i)}{\varphi_0(\t_i)}\right)^q\right)^{\frac{p}{q}}
\left(\sup_{\t_i\in M_k^0}
\frac{|c_i|}{\varphi_1(\t_i)}\right)^p\right)^{1/p}
\\
 &\hskip+1cm \ls
  \left(\sum_{k\in\Z}\sum_{\tau_j\in M_k^1}
\left(
\sum_{\t_i\in M_j^0, \frac{\varphi_1(\t_i)}{\varphi_0(\t_i)}\le t}
\frac{\varphi_1(\t_i)}{\varphi_0(\t_i)}\right)^{p}
\left(\sup_{\t_i\in M_j^0}
\frac{|c_i|}{\varphi_1(\t_i)}\right)^p \right)^{1/p}\\
&\hskip+1cm\ls
 \left(\sum_{k\in\Z}\sum_{\tau_j\in M_k^1}
\left(
\sup_{\t_i\in M_j^0, \frac{\varphi_1(\t_i)}{\varphi_0(\t_i)}\le t}
\frac{\varphi_1(\t_i)}{\varphi_0(\t_i)}\right)^{p}
\left(\sup_{\t_i\in M_j^0}
\frac{|c_i|}{\varphi_1(\t_i)}\right)^p \right)^{1/p}\\
&\hskip+1cm\ls t
 \left(\sum_{k\in\Z}\sup_{\tau_j\in M_k^1}
\left(\sup_{\t_i\in M_j^0}
\frac{|c_i|}{\varphi_1(\t_i)}\right)^p \right)^{1/p}\\
 &\hskip+1cm \ls 
  t \left(\sum_{k\in\Z}
\left(\sum_{\tau_i\in M_k^1}
\left(\frac{|c_i|}{\varphi_1(\t_i)}\right)^q\right)^{p/q}\right)^{1/p}.
\end{align*} 
Therefore, we get 
\begin{align}  \label{I}
&\|a_i\chi_{\Omega_t}(\t_i)\|_{l_p\left(l_q^{M_k^0}\left(\frac{1}{\varphi_0(\t_i)}\right)\right)} \\
&\hskip+1cm\ls  
\|b_i\|_{l_p\left(l_q^{M_k^0}\left(\frac{1}{\varphi_0(\t_i)}\right)\right)}+t  \|c_i\|_{l_p\left(l_q^{M_k^1}\left(\frac{1}{\varphi_0(\t_i)}\right)\right)}.\notag
\end{align}

Similarly by using  estimates \eqref{sum<max2} and \eqref{sum<max1} we obtain 
\begin{align*}
&\left(\sum_{k\in\Z}\left(
\sum_{\t_i\in M_k^1, \frac{\varphi_1(\t_i)}{\varphi_1(\t_i)}> t}
\left(\frac{|b_i|}{\varphi_1(\t_i)}\right)^q\right)^{\frac{p}{q}}\right)^{1/p} \\
&\hskip+1cm \ls t\left(\sum_{k\in\Z}
\left(
\sum_{\t_i\in M_k^1, \frac{\varphi_1(\t_i)}{\varphi_0(\t_i)}> t}
\left(\frac{\varphi_0(\t_i)}{\varphi_1(\t_i)}\right)^q\right)^{\frac{p}{q}}
\left(\sup_{\t_i\in M_k^1}
\frac{|b_i|}{\varphi_0(\t_i)}\right)^p\right)^{1/p}
\\
  &\hskip+1cm\ls t\left(\sum_{k\in\Z} 
  \sum_{z_j\in M_k^0}
\left(
\sum_{\t_i\in M_j^1, \frac{\varphi_1(\t_i)}{\varphi_0(\t_i)} > t}
\frac{\varphi_0(\t_i)}{\varphi_1(\t_i)}\right)^{p}
\left(\sup_{\t_i\in M_j^1}
\frac{|b_i|}{\varphi_0(\t_i)}\right)^p \right)^{1/p}\\
 &\hskip+1cm\ls t\left(\sum_{k\in\Z}  \sum_{z_j\in M_k^0}
\left(
\sup_{\t_i\in M_j^1, \frac{\varphi_0(\t_i)}{\varphi_1(\t_i)}< 1/t}
\frac{\varphi_0(\t_i)}{\varphi_1(\t_i)}\right)^{p}
\left(\sup_{\t_i\in M_j^1}
\frac{|b_i|}{\varphi_0(\t_i)}\right)^p \right)^{1/p}\\
 &\hskip+1cm\ls \left(\sum_{k\in\Z}  \sup_{z_j\in M_k^0}
\left(\sup_{\t_i\in M_j^1}
\frac{|b_i|}{\varphi_0(\t_i)}\right)^p \right)^{1/p}\\
 &\hskip+1cm\ls  \left(\sum_{k\in\Z}
\left(\sup_{\t_i\in M_k^0}
\left(\frac{|b_i|}{\varphi_0(\t_i)}\right)^{q}\right)^{p/q}\right)^{1/p}.
\end{align*} 
Hence
\begin{align}\label{II}
 &t\|a_i\chi_{\Omega_t^c}(\t_i)\|_{l_p\left(l_q^{M_k^1}\left(\frac{1}{\varphi_1(\t_i)}\right)\right)} \\
&\hskip+1cm\ls 
\|b_i\|_{l_p\left(l_q^{M_k^0}\left(\frac{1}{\varphi_0(\t_i)}\right)\right)}+t  \|c_i\|_{l_p\left(l_q^{M_k^1}\left(\frac{1}{\varphi_0(\t_i)}\right)\right)}.\notag
\end{align}
Combining the estimates \eqref{I} and  \eqref{II} we  get
\begin{align*}
&\|a_i\chi_{\Omega_t}(t_i)\|_{l_p\left(l_q^{M_k^0}\left(\frac{1}{\varphi_0(\t_i)}\right)\right)}+t\|a_i\chi_{\Omega_t^c}(t_i)\|_{l_p\left(l_q^{M_k^1}\left(\frac{1}{\varphi_1(\t_i)}\right)\right)}\\
&\hskip+1cm \ls \|b_i\|_{l_p\left(l_q^{M_k^0}\left(\frac{1}{\varphi_0(\t_i)}\right)\right)}+ t\|c_i\|_
{l_p\left(l_q^{M_k^1}\left(\frac{1}{\varphi_1(\t_i)}\right)\right)}.
\end{align*}  
If we  take the infimum over all the representations $a_i=b_i+c_i$ we obtain the desired lower bound estimate for the $K$-functional.
\end{proof}

\begin{lem}\label{lppp} Let $q,p \in[1,\infty]$. Then
\begin{equation} \label{ppp}
((\l_q)_{\varphi_0,p},(\l_q)_{\varphi_1,p})_{\varphi,p}=l_p\left(\frac{1}{\varphi(\varphi_0,\varphi_1)(\t_i)}\right).
\end{equation}
\end{lem}
\begin{proof}
It is enough to show that
\begin{equation}\label{ppp1}
\left((\l_1)_{\varphi_0,p},(\l_1)_{\varphi_1,p}\right)_{\varphi,p}=
\left((\l_\infty)_{\varphi_0,p},(\l_\infty)_{\varphi_1,p}\right)_{\varphi,p}.
 \end{equation}  
 Indeed, from the embeddings 
 $$ l_1\subset  l_q \subset l_\infty \quad \text{and} \quad l_1\left(\frac 1{\t_k}\right)\subset  l_q \left(\frac 1{\t_k}\right)\subset l_\infty\left(\frac 1{\t_k}\right), $$
  we get 
  \begin{align}\label{l1lq}
\left((\l_1)_{\varphi_0,p},(\l_1)_{\varphi_1,p}\right)_{\varphi,p}&\subset
\left((\l_q)_{\varphi_0,p},(\l_q)_{\varphi_1,p}\right)_{\varphi,p}\\
 &\subset  
\left((\l_\infty)_{\varphi_0,p},(\l_\infty)_{\varphi_1,p}\right)_{\varphi,p},\notag
 \end{align}  
 consequently, using \eqref{ppp1}, \eqref{l1lq}, and Corollary~\ref{cor3}, we obtain
 \begin{align*}
  \left((\l_q)_{\varphi_0,p},(\l_q)_{\varphi_1,p}\right)_{\varphi,p}
  &=
  \left((\l_p)_{\varphi_0,p},(\l_p)_{\varphi_1,p}\right)_{\varphi,p}\\
  &=l_p\left(\frac{1}{\varphi(\varphi_0,\varphi_1)(\t_i)}\right).
  \end{align*}
  To show \eqref{ppp1} we only need to prove the embedding
\begin{equation}\label{ppple}
\left((\l_\infty)_{\varphi_0,p},(\l_\infty)_{\varphi_1,p}\right)_{\varphi,p}\subset
\left((\l_1)_{\varphi_0,p},(\l_1)_{\varphi_1,p}\right)_{\varphi,p}.
 \end{equation}  
By Lemma~\ref{kfunc}
 \begin{align*}
&\|\{a_i\}\|_{\left((\l_1)_{\varphi_0,p},(\l_1)_{\varphi_1,p}\right)_{\varphi,p}}^p\\
 &\hskip+1cm \approx \sum_{j\in\Z}
   \sum_{k\in\Z}\left(
\sum_{\t_i\in M_k^0, \frac{\varphi_1(\t_i)}{\varphi_0(\t_i)}\le t_j}
\frac{|a_i|}{\varphi_0(\t_i)}\right)^{p} \frac{1}{\varphi(t_j)^p}\\
&\hskip+1.5cm +\sum_{j\in\Z}
\sum_{k\in\Z}\left(
\sum_{\t_i\in M_k^1, \frac{\varphi_1(\t_i)}{\varphi_0(\t_i)}\ge t_j}
\frac{|a_i|}{\varphi_1(\t_i)}\right)^{p} \left(\frac{t_j}{\varphi(t_j)}\right)^p=I+II,
\end{align*}  
and 
 \begin{align*}
&\|\{a_i\}\|_{\left((\l_\infty)_{\varphi_0,p},(\l_\infty)_{\varphi_1,p}\right)_{\varphi,p}}^p\\
 &\hskip+1cm \approx \sum_{j\in\Z}
   \sum_{k\in\Z}\left(
\sup_{\t_i\in M_k^0, \frac{\varphi_1(\t_i)}{\varphi_0(\t_i)}\le t_j}
\frac{|a_i|}{\varphi_0(\t_i)}\right)^{p} \frac{1}{\varphi(t_j)^p}\\
&\hskip+1.5cm +\sum_{j\in\Z}
\sum_{k\in\Z}\left(
\sup_{\t_i\in M_k^1, \frac{\varphi_1(\t_i)}{\varphi_0(\t_i)}\ge t_j}
\frac{|a_i|}{\varphi_1(\t_i)}\right)^{p} \left(\frac{t_j}{\varphi(t_j)}\right)^p=III+IV.
\end{align*}  
 Let  $A_k^n$ be defined $A_k^n:=\sup_{\t_i\in M_k^n} \frac{\varphi_1(\t_i)}{\varphi_0(\t_i)}, \quad n=0,1$. Using 
 \eqref{sum<max3} and  \eqref{sum<max3} we obtain
   \begin{align*}
I &\ls  
   \sum_{k\in\Z}
   \left(
\sum_{\t_i\in M_k^0, \frac{\varphi_1(\t_i)}{\varphi_0(\t_i)}\le t_j}
\frac{|a_i|}{\varphi_0(\t_i)}\right)^{p} \sum_{A_k^0 \le t_j} \frac{1}{\varphi(t_j)^p}\\
&\ls
 \sum_{k\in\Z}\sum_{\t_i\in M_k^0}
\left(\frac{|a_i|}{\varphi(\varphi_0,\varphi_1)(\t_i)}\right)^{p} \times\\
 &\hskip+2cm \times
\left(\sum_{\t_i\in M_k^0, \frac{\varphi_1(\t_i)}{\varphi_0(\t_i)}\le t_j}\varphi\left(\frac{\varphi_1(\t_i)}{\varphi_0(\t_i)}\right)^{p'}\right)^{p/p'}
\frac{1}{\varphi(A_k^0)^p}\\
&\ls
 \sum_{k\in\Z}\sum_{\t_i\in M_k^0}
\left(\frac{|a_i|}{\varphi(\varphi_0,\varphi_1)(\t_i)}\right)^{p} 
\left(\sup_{\t_i\in M_k^0}\varphi\left(\frac{\varphi_1(\t_i)}{\varphi_0(\t_i)}\right)\right)^{p} 
\frac{1}{\varphi(A_k^0)^p}\\
&\ls\sum_{k\in\Z}\sum_{\t_i\in M_k^0}
\left(\frac{|a_i|}{\varphi(\varphi_0,\varphi_1)(\t_i)}\right)^{p}
  \\
&=\sum_{j\in\Z_1} \sum_{k\in\Z}\sum_{\t_i\in M_k^0, t_j<\frac{\varphi_1(\t_i)}{\varphi_0(\t_i)}\le t_{j+1}}
\left(\frac{|a_i|}{\varphi(\varphi_0,\varphi_1)(\t_i)}\right)^{p}\\
&\hskip0.5cm +\sum_{j\in\Z_2} \sum_{k\in\Z}\sum_{\t_i\in M_k^1, t_j<\frac{\varphi_1(\t_i)}{\varphi_0(\t_i)}\le t_{j+1}}
\left(\frac{|a_i|}{\varphi(\varphi_0,\varphi_1)(\t_i)}\right)^{p}\\
&\ls \sum_{j\in\Z_1} \sum_{k\in\Z}\left(\sup_{\t_i\in M_k^0, t_j<\frac{\varphi_1(\t_i)}{\varphi_0(\t_i)}\le t_{j+1}}
\frac{|a_i|}{\varphi_0(\t_i)}\right)^{p}\times\\
&\hskip+2cm \times
\sum_{\t_i\in M_k^0, t_j<\frac{\varphi_1(\t_i)}{\varphi_0(\t_i)}\le t_{j+1}}
\varphi\left(\frac{\varphi_1(\t_i)}{\varphi_0(\t_i)}\right)^{-p}
\\
&+\hskip+0.5cm\sum_{j\in\Z_2} \sum_{k\in\Z}\left(\sup_{\t_i\in M_k^1, t_j<\frac{\varphi_1(\t_i)}{\varphi_0(\t_i)}\le t_{j+1}}
\frac{|a_i|}{\varphi_1(\t_i)}\right)^{p}\times\\
&\hskip+2cm \times
\sum_{\t_i\in M_k^0, t_j<\frac{\varphi_1(\t_i)}{\varphi_0(\t_i)}\le t_{j+1}}
\left(\frac{\varphi_0(\t_i)}{\varphi_1(\t_i)}\varphi\left(\frac{\varphi_1(\t_i)}{\varphi_0(\t_i)}\right)\right)^{-p}\\
 &\ls \sum_{j\in\Z_1} \sum_{k\in\Z}\left(\sup_{\t_i\in M_k^0, t_j<\frac{\varphi_1(\t_i)}{\varphi_0(\t_i)}\le t_{j+1}}
\frac{|a_i|}{\varphi_0(\t_i)}\right)^{p}
\sup_{ t_j<\frac{\varphi_1(\t_i)}{\varphi_0(\t_i)}\le t_{j+1}}
\varphi\left(\frac{\varphi_1(\t_i)}{\varphi_0(\t_i)}\right)^{-p}
\\
&\hskip+0.5cm +\sum_{j\in\Z_2} \sum_{k\in\Z}\left(\sup_{\t_i\in M_k^1, t_j<\frac{\varphi_1(\t_i)}{\varphi_0(\t_i)}\le t_{j+1}}
\frac{|a_i|}{\varphi_1(\t_i)}\right)^{p}\times\\
&\hskip+2cm \times\sup_{ t_j<\frac{\varphi_1(\t_i)}{\varphi_0(\t_i)}\le t_{j+1}}
\left(\frac{\varphi_0(\t_i)}{\varphi_1(\t_i)}\varphi\left(\frac{\varphi_1(\t_i)}{\varphi_0(\t_i)}\right)\right)^{-p}\\
&\ls \sum_{j\in\Z_1} \sum_{k\in\Z}\left(\sup_{\t_i\in M_k^0, t_j<\frac{\varphi_1(\t_i)}{\varphi_0(\t_i)}\le t_{j+1}}
\frac{|a_i|}{\varphi_0(\t_i)}\right)^{p}
\frac{1}{\varphi(t_j)^{p}}
\\
&\hskip+0.5cm +\sum_{j\in\Z_2} \sum_{k\in\Z}\left(\sup_{\t_i\in M_k^1, t_j<\frac{\varphi_1(\t_i)}{\varphi_0(\t_i)}\le t_{j+1}}
\frac{|a_i|}{\varphi_1(\t_i)}\right)^{p}
\left(\frac{t_j}{\varphi(t_j)}\right)^{p}
=III+ IV.
\end{align*} 
Similarly we see that
$$ II\ls III+ IV.$$
Finally, combining estimates we get \eqref{ppple}.
The proof is complete. 
\end{proof}

\begin{thm}   Let $p \in[1,\infty]$. Then
$$(\X_{\varphi_0,p},\X_{\varphi_1,p})_{\varphi,p}=\X_{\varphi(\varphi_0,\varphi_1),p}.
$$
\end{thm}
\begin{proof} Using \eqref{orb} and \eqref{ppp1} we get 
\begin{align*}
&\X_{\varphi(\varphi_0,\varphi_1),p} \\
 &\hskip+0.2cm =
\Orb\left(l_p\left(\frac{1}{\varphi(\varphi_0,\varphi_1)(\t_i)}\right): \; \l_1 \rightarrow \X    \right)\\
  &\hskip+0.2cm  \subset \Orb\left(l_p\left(\frac{1}{\varphi(\varphi_0,\varphi_1)(\t_i)}\right):\; \{(\l_1)_{\varphi_0,p},(\l_1)_{\varphi_1,p}\}\rightarrow
  \{\X_{\varphi_0,p},\X_{\varphi_1,p}\}\right)\\
 &\hskip+1cm \subset (\X_{\varphi_0,p},\X_{\varphi_1,p})_{\varphi,p}.
 \end{align*}
 Similarly,  using \eqref{orb} and \eqref{l1p} we get  
 \begin{align*}
 &(\X_{\varphi_0,p},\X_{\varphi_1,p})_{\varphi,p}\\
   &\hskip+0.2cm  \subset \Corb\left(l_p\left(\frac{1}{\varphi(\varphi_0,\varphi_1)(\t_i)}\right): \;   \{\X_{\varphi_0,p},\X_{\varphi_1,p}\} \rightarrow \{(\l_\infty)_{\varphi_0,p},(\l_\infty)_{\varphi_1,p}\}\right)\\
&\hskip+0.2cm \subset \Corb\left(l_p\left(\frac{1}{\varphi(\varphi_0,\varphi_1)(\t_i)}\right): \;  \X \rightarrow \l_\infty\right)\\
&\hskip+0.2cm =\X_{\varphi(\varphi_0,\varphi_1),p}.
\end{align*}
The proof is complete.
\end{proof}

\section{ Proof of Theorem~\ref{main}}
$\rm{(v)} \Rightarrow \rm{(iv)}$. If (v) holds, by Corollary~\ref{cor1} and \ref{cor2} we have 
\begin{align}\label{l1l1p}
\left((\l_1)_{\varphi_0,1},(\l_p)_{\varphi_1,p}\right)_{\varphi,p} &=
\left((\l_\infty)_{\varphi_0,1},(\l_\infty)_{\varphi_1,p}\right)_{\varphi,p} \\
&=
l_p\left(\frac{1}{\varphi(\varphi_0,\varphi_1)(\t_i)}\right).\notag
\end{align}
Using \eqref{orb} and \eqref{l1l1p}we get 
\begin{align*}
&\X_{\varphi(\varphi_0,\varphi_1),p} \\
 &\hskip+0.2cm =
\Orb\left(l_p\left(\frac{1}{\varphi(\varphi_0,\varphi_1)(\t_i)}\right): \; \l_1 \rightarrow \X    \right)\\
  &\hskip+0.2cm  \subset \Orb\left(l_p\left(\frac{1}{\varphi(\varphi_0,\varphi_1)(\t_i)}\right):\; \{(\l_1)_{\varphi_0,1},(\l_1)_{\varphi_1,1}\}\rightarrow
  \{\X_{\varphi_0,1},\X_{\varphi_1,1}\}\right)\\
 &\hskip+0.2cm \subset (\X_{\varphi_0,1},\X_{\varphi_1,1})_{\varphi,p}.
 \end{align*}
 Similarly,  using \eqref{orb} and \eqref{l1p} we get  
 \begin{align*}
 &(\X_{\varphi_0,\infty},\X_{\varphi_1,\infty})_{\varphi,p}\\
   & \subset \Corb\left(l_p\left(\frac{1}{\varphi(\varphi_0,\varphi_1)(\t_i)}\right): \;   \{\X_{\varphi_0,\infty},\X_{\varphi_1,\infty}\} \rightarrow \{(\l_\infty)_{\varphi_0,\infty},(\l_\infty)_{\varphi_1,\infty}\}\right)\\
&\hskip+0.2cm \subset \Corb\left(l_p\left(\frac{1}{\varphi(\varphi_0,\varphi_1)(\t_i)}\right): \; \X \rightarrow \l_\infty\right)=\X_{\varphi(\varphi_0,\varphi_1),p}.
\end{align*}
Since 
$$ (\X_{\varphi_0,1},\X_{\varphi_1,1})_{\varphi,p}
 \subset (\X_{\varphi_0,\infty},\X_{\varphi_1,\infty})_{\varphi,p},
$$
we obtain 
\begin{align*} 
(\X_{\varphi_0,p_0},\X_{\varphi_1,p_1})_{\varphi,p}
&= (\X_{\varphi_0,p_0},\X_{\varphi_1,p_1})_{\varphi,p}\\
&=(\X_{\varphi_0,\infty},\X_{\varphi_1,\infty})_{\varphi,p}\\
&=\X_{\varphi(\varphi_0,\varphi_1),p}.
\end{align*}
We now complete the  proof of the implication $\rm{(v)} \Rightarrow \rm{(iv)}$. 

The implications $ \rm{(iv)} \Rightarrow \rm{(ii)} \Rightarrow \rm{(i)}\Rightarrow \rm{(iii)}$ are clear. We will show the implication  $\rm{(iii)} \Rightarrow \rm{(v)}$.

Suppose that \eqref{stab} holds for the couple $\X=(L_1(0,\infty), L_\infty(0,\infty))$ and for some $p \in (0,\infty)$. As the couple
$(L_1(0,\infty), L_\infty(0,\infty))$  is complete  couple then \eqref{stab}  holds for any  couple (see \cite{BK}) and  therefore for the couple $\overline{l}_q=(l_q, l_q(\frac{1}{\t_k}))$.
By Corollary~\ref{cor1}, \eqref{lqp} and \eqref{ppp1} we get
\begin{align*}
l_p\left(l_q^{M_k}\left(\frac{1}{\varphi(\varphi_0,\varphi_1)(\t_i)}\right)\right)
&=\left((\l_q)_{\varphi_0,q},(\l_q)_{\varphi_1,q}\right)_{\varphi,p} \\ 
&=\left((\l_q)_{\varphi_0,p},(\l_q)_{\varphi_1,p}\right)_{\varphi,p}\\ 
&=(\l_q)_{\varphi(\varphi_0,\varphi_1),p}\\
&=l_p\left(\frac{1}{\varphi(\varphi_0,\varphi_1)(\t_i)}\right).
\end{align*}
It is easy to see that from here  that (v) follows.
The proof is complete.

\textbf{Acknowledgment:} I thank the referees for their valuable comments.

\end{document}